\documentclass[a4paper,12pt]{article}
\usepackage[utf8x]{inputenc}
\usepackage{latexsym,amssymb}
\usepackage{xcolor}
\usepackage[pdftex]{graphicx}
\usepackage{amsmath,amsfonts,amsthm}
\usepackage{multicol} 
\usepackage{pifont}
\usepackage{geometry}

\newtheorem{thm}{Theorem}

\newtheorem{lemma}{Lemma}
\newtheorem{remark}{Remark}

\newcommand{\Q}{{\mathbb Q}}

\newcommand{\Z}{{\mathbb Z}}

\newcommand{\Ha}{{\mathbb H}}

\newcommand{\Oo}{{\mathcal O}}
\newcommand{\afr}{\mathfrak{a}}
\newcommand{\eps}{\varepsilon}

\newcommand{\al}{\alpha}

\newcommand{\Addresses}{{
  \bigskip
  \footnotesize

  Sonia Samol, \textsc{Institut f\"ur Mathematik, Johannes Gutenberg Universit\"at Mainz,
    Germany}\par\nopagebreak
  \textit{E-mail address} \texttt{samols@uni-mainz.de}

}}

\begin{document}
\title{Effective bounds for the negativity of Shimura curves on Hilbert Modular Surfaces }
\author{Sonia Samol}
\date{}
\maketitle
\begin{abstract} We study the bounded negativity conjecture for non-quaternionic Hilbert modular surfaces and give an explicit bound for the special case of Hirzebruch-Zagier curves on Hilbert modular surfaces.\end{abstract}

\section*{Introduction}
The bounded negativity conjecture states that for each smooth complex projective surface $X$ there exists a number $b(X)\geq 0$ such that $C^2\geq - b(X)$ for every reduced, irreducible curve $C \subset X$.\\
It is unclear where the origins of it conjecture lay, but it is believed to have been already mentioned by Federigo Enriques to his last student, Alfredo Franchetta. \\

In recent years, there has been made a lot of progress on verifying the conjecture in special cases. For this paper, the most important result was proven in 2011: In [1] it was shown that for example for Shimura curves on compact Hilbert modular surfaces the conjecture is true:\\

\begin{thm} (Bauer, Harbourne, Knutsen, K\"uronya, M\"uller-Stach, Roulleau, Szemberg).\\
For a Shimura curve $C$, not necessarily smooth, on a quaternionic Shimura surface of Hilbert modular type $X$ the inequality
\begin{equation*}
 C^2\geq -6c_2(X)
\end{equation*}
holds, where $c_2(X)$ is the second Chern class of $X$. Moreover, there is only a finite number of Shimura curves with $C^2<0$.\end{thm}
This theorem gave rise to the problems we consider in this paper: Firstly, whether we can give an explicit bound for the negative self-intersection numbers of Hirzebruch-Zagier curves on Hilbert modular surfaces and secondly how the bound given in Theorem $1$ changes when we consider non-compact Hilbert modular surfaces.\\
In the first part of this paper we will look at those Hirzebruch-Zagier curves on Hilbert modular surfaces and give an effective bound for their negativity.\\
Let $p\equiv 1$ mod $4$ be a prime, $\Oo_K$ the ring of integers of $K=\Q(\sqrt{p})$. The group SL$_2(\Oo_K)$ acts proper discontinuously on $\Ha^2$ the product of the upper half plane with itself by
\[\left(\begin{matrix}\alpha & \beta \\  \gamma & \delta \end{matrix}\right)(z_1,z_2)\mapsto \left(\frac{\alpha z_1+\beta}{\gamma z_1 + \delta}, \frac{\alpha' z_2+ \beta'}{\gamma'z_2+\delta'} \right).\]For a given $N$ the Hirzebruch-Zagier curves are defined as all points $(z_1,z_2)\in \Ha^2$ satisfying the equation 
\begin{equation*}a\sqrt{p}z_1z_2+\lambda z_2-\lambda'z_1+b\sqrt{p}=0
\end{equation*} with $a,b \in \Z, \lambda \in \Oo, \lambda \lambda'+abp=N$.\\ In 1976, Friedrich Hirzebruch and Don Zagier proved in \cite{SS_HZ76} that the generating series for the intersection numbers of the
Hirzebruch-Zagier cycles $T_N$ is a classical modular form of weight 2 by giving an explicit formula for these intersection numbers:
\begin{thm} (Hirzebruch-Zagier \cite{SS_HZ76}, Theorem 4):
\begin{equation*}
 {T_N^2}= \frac{1}{2} \sum_{n|N}n H_p\left(\frac{N^2}{n^2}\right)\left(\chi_{p}(n)+\chi_{p}\left(\frac{NA}{n}\right)\right),
\end{equation*}

where $H_p(n)$ is some sum of class numbers of positive definite primitive binary integral quadratic forms.\end{thm}
We use this formula to find a bound for the negativity of the Hirzebruch-Zagier curves. To that purpose we give estimates for the appearing terms with the aim of ending up with a polynomial of degree $2$ in $N$, the discriminant of the Hirzebruch-Zagier curves. We deal with this in lemmata $1$ to $3$ of this paper, then the determination of the minimum is an easy polynomial extremum calculation.\\ 
First, we restrict to the case $N=\prod_{i=1}^k p_i$ with $p_1,..,p_k$ pairwise different prime numbers with $\chi_p(p_i)=1$, because otherwise the self-intersection numbers will be more positive and hence not interesting for the bounded negativity conjecture. But then the sum simplifies to 
\begin{equation*}
 {T_N^2}=  \sum_{n|N}n H_p\left(\frac{N^2}{n^2}\right).
\end{equation*}
The sum can be split into a negative and a positive part where the negative part is just $-\frac{1}{6}\sigma_1(N)$ and can be measured by Robin's estimate for the divisor sum.\\
For the positive part we get a sum over some class numbers. The sum is defined over solutions to the equation \begin{equation*}x \in \mathbb{Z}, x^2 \leq 4n, x^2 \equiv 4n \textnormal{ mod } p. \end{equation*}
We approximate the solutions and insert Payley's inequality for the class numbers in Lemma $1$ to get an estimate for the self-intersection numbers.
The only remaining obstacle to detecting the minimum of the self-intersection numbers lies now in the logarithmic terms which occur both in Payley's estimate for the class numbers and in Robin's estimate for the divisor sum. But we solve this issue by replacing them with their biggest possible value in the intervall we are interested in.\\
 
Therefore we get an effective version of Theorem 1 for the Hirzebruch-Zagier curves $T_N$ on Hilbert modular surfaces, namely

\begin{thm}For $N=\prod_{i=1}^k p_i$ with $p_1,..,p_k$ pairwise different prime numbers with $\chi_p(p_i)=1$, we have for $p \rightarrow \infty$
\begin{align*}T_N^2 \geq -\left(\frac{1}{96}\frac{c^2}{\delta}\right)  p^{\frac{3}{2}},
\end{align*}
with $\delta:=\frac{\pi}{12\textnormal{e}^{\gamma}}$ and $c:=\textnormal{e}^{\gamma}+0.6482$ with $\gamma$ the Euler-Mascheroni constant.
\end{thm}
In the second part we follow closely the approach in [1], using results by Miyaoka to generalize Theorem 1 for non-quaternionic Hilbert modular surfaces.
\begin{thm}
For a Shimura curve $C$ on a non-quaternionic Hilbert modular surface $X$ we have 
\begin{align*}C^2\geq -9d_2(X),\end{align*}
where $d_2(X):=3c_2(X)-K_X^2$.
\end{thm}
For the special case of a Hilbert modular surface over $K=\mathbb{Q}(\sqrt{p})$ that we considered in the first part, a bound for the second Chern class can be calculated using \cite{{SS_GvdG88}}.

So by this approach we get the bound
\begin{align*}
C^2 &\geq -\left(\frac{9}{10}p^{\frac{3}{2}}\left(\frac{3}{2\pi^2}\log^2(p)+1.05\log p\right)+ \frac{27}{2}p^{\frac{1}{2}}\left(\frac{3}{2\pi^2}\log^2(p)+1.05\log p\right)\right.\\
&+\left.\frac{27 \pi}{8\textnormal{e}^{\gamma}}\frac{p^{\frac{1}{2}}}{\log\log(4p)}+\frac{3\cdot 5\sqrt{3}\pi}{2\textnormal{e}^{\gamma}}\frac{p^{\frac{1}{2}}}{\log\log(3p)}+\frac{3 \cdot \sqrt{3}\pi}{2\textnormal{e}^{\gamma}}\frac{p^{\frac{1}{2}}}{\log\log(3p)}\right).
\end{align*}

\section{A bound for Hirzebruch-Zagier curves on Hilbert modular surfaces}

First we give some notations and definitions.\\
Let $p \equiv 1$ mod $4$ be a prime, $K=\mathbb{Q}(\sqrt{p})$, $\mathcal{O}$ the ring of integers of $K$ and $\mathfrak{a}$ an ideal in $\mathcal{O}$
 with Norm($\mathfrak{a}$)$=A$.\\
The group SL$_2(\mathcal{O},\mathfrak{a})=\left\{ T \in \left(\begin{matrix} 
                                                     \mathcal{O} & \mathfrak{a}^{-1}\\
                                                     \mathfrak{a} & \mathcal{O}
                                                    \end{matrix}
\right), \det T =1 \right\}$ acts proper discontinuously on $\Ha^2$ the product of the upper half plane with itself by
\[\left(\begin{matrix}\alpha & \beta \\  \gamma & \delta \end{matrix}\right)(z_1,z_2)\mapsto \left(\frac{\alpha z_1+\beta}{\gamma z_1 + \delta}, \frac{\alpha' z_2+ \beta'}{\gamma'z_2+\delta'} \right).\] 
In the special case $\afr=\Oo$ we get SL$_2(\mathcal{O},\Oo)=$SL$_2(\mathcal{O})$.\\
The quotient $X^{\mathfrak{a}}=\mathbb{H}^2/ SL_2(\mathcal{O},\mathfrak{a})$ is a non-compact complex surface with finitely many singularities (namely $h(p)$ many, where $h(\cdot)$ denotes the class number of positive definite primitive binary integral quadratic forms with discriminant $p$) which can be compactified
by adding the cusps to $X^{\mathfrak{a}}$ and resolving the singularities created. Then one gets the Hirzebruch compactification $\bar{X^{\mathfrak{a}}}=X^{\mathfrak{a}} \cup \bigcup_k S_k$ where the $S_k$ are rational curves.\\

\noindent\textbf{Definition:} A skew-hermitian matrix 
\begin{equation*}
 B=\left(\begin{matrix} a\sqrt{D} & \lambda \\
            -{\lambda}^{'} & \frac{b}{A} \sqrt{D}
        
       \end{matrix}\right)
\end{equation*}
is called $\mathfrak{a}-$integral if $a$ and $b$ are integrals and $\lambda \in \mathfrak{a}^{-1}$, where $\lambda'$ is the conjugate of $\lambda$.\\
If there is no integer $n>1$ with $(\frac{a}{n},\frac{b}{n}, \frac{\lambda}{n})\in \mathbb{Z}^2\times \mathfrak{a}^{-1}$, then $ B$ is called primitive.\\

\noindent\textbf{Definition:} For a primitive, $\mathfrak{a}-$integral, skew-hermitian matrix $B$ the curve $F_B$ is defined as the image of the set 
\begin{equation*}
 \left\{ (z_1,z_2) \in \mathbb{H}^2\cup \mathbb{P}^1(K): (z_2, 1) B \left(\begin{matrix} 
                                                     z_1 \\
                                                     1
                                                    \end{matrix}
\right)=0\right\}
\end{equation*}
 in $X^{\mathfrak{a}}$. With Norm($\mathfrak{a}$)$=A$ the curve $F_N$ is defined as 

\begin{equation*}
 F_N:= \bigcup_{\begin{matrix} 
                                                     B \textnormal{ as above } \\
                                                     \det(B)=\frac{N}{A}
                                                    \end{matrix}
} F_B.
\end{equation*}

\noindent Franke showed in \cite{SS_HF78} that for a prime discriminant $p$ and $p^2 \nmid N$ the curve $ F_N$ consists of only one component.\\
The Hirzebruch-Zagier curve $T_N$ is then defined as 
\begin{equation*}
 T_N = \bigcup_{\begin{matrix} 
                                                     t \geq 1 \\
                                                     t^2|N
                                                    \end{matrix}}F_{\frac{N}{t^2}},
\end{equation*}
so for $N$ squarefree one gets $T_N=F_N$ irreducible. Furthermore, $F_N$ is not empty if $\chi_p(NA) \neq -1$, where $\chi_p(n)=\left(\frac{n}{p}\right) $ is the Legendre symbol, and compact 
if $N$ is not the norm of an ideal in the genus of $\mathfrak{a}$.\\

For the self-intersection number of the curves $T_N$ Hirzebruch and Zagier showed the following formula:\\

\begin{thm} (Hirzebruch-Zagier \cite{SS_HZ76}, Theorem 4):
\begin{equation*}
 {T_N^2}= \frac{1}{2} \sum_{n|N}n \left(H_p\left(\frac{N^2}{n^2}\right)+I_p\left(\frac{N^2}{n^2}\right)\right)\left(\chi_{p}(n)+\chi_{p}\left(\frac{NA}{n}\right)\right),
\end{equation*}
with
\begin{equation*}H_p(n)=\sum_{\substack{x \in \mathbb{Z}\\x^2 \leq 4n\\ x^2 \equiv 4n \textnormal{ mod } p}}H \left(\frac{4n-x^2}{p}\right),
\end{equation*}
\begin{equation*}
 H(n)=\left\{\begin{array}{cc}-\frac{1}{12} & \textnormal{ if } n=0 \\ \sum_{d^2|n}h'\left(-\frac{n}{d^2}\right) & \textnormal{else}\end{array}\right.,
\end{equation*}
\begin{equation*}
 h'(\Delta)= \left\{\begin{array}{cc}\frac{1}{3} & \textnormal{ if } \Delta=-3\\
                     \frac{1}{2} & \textnormal{ if } \Delta=-4\\
                     h(\Delta) & \textnormal{ if } \Delta \equiv 0 \textnormal{ or } 1 \textnormal{ mod } 4, \Delta \leq -4 
                    
\end{array}\right.,
\end{equation*}
where $h(\Delta)$ is defined as the class number of positive definite primitive binary integral quadratic forms with discriminant $\Delta$ and
\begin{equation*}
 I_p(n)=\frac{1}{\sqrt{p}} \sum_{\substack{\lambda \in \mathcal{O}\\ \lambda>0, \lambda'>0\\ \lambda \lambda'=n }}\min(\lambda,\lambda').
\end{equation*}\end{thm}

\noindent We use this formula to first show that there exists only finitely many curves $T_N$ with negative self-intersection and then use this to find a bound $b(X^{\mathfrak{a}})$ s.t. $T_N^2\geq -b(X^{\mathfrak{a}})$ for $\chi_p(A)=1$ with norm$({\mathfrak{a}})=A$ .\\

\noindent Let $N=\prod_{i=1}^k p_i$ with $p_1,..,p_k$ (pairwise different, if irreducible is wanted) prime numbers with $\chi_p(p_i)=1$, $\sigma_1(n):=\sum_{d|n}d$. As in \cite{SS_HZ76} we write $H_p(n)=-\frac{1}{6}n+H_p^0(n)$, where $H_p^0(n)=\sum_{\substack{x \in \mathbb{Z}\\x^2 < 4n\\ x^2 \equiv 4n \textnormal{ mod } p}}H \left(\frac{4n-x^2}{p}\right)$ has only summands bigger or equal to zero.\\
Since we look at a sum over class numbers, we need an estimate for those numbers, but Payley showed in \cite{Payley} combined with Littlewood \cite{Littlewood} that for class numbers of a negative discriminant we have $h(-d)\geq \frac{\pi}{24\textnormal{e}^{\gamma}}\frac{\sqrt{d}}{\log\log(d)}=:\tilde{h}(-d)$, where $\gamma:=\lim_{n \rightarrow \infty}\sum_{k=1}^n(\frac{1}{k}-\log n)$ is the Euler-Mascheroni constant. We use this estimate to get a better grip on the positive part $H_p^0(n^2)$. By defining $H'(n):=\sum_{d^2|n}\tilde{h}\left(-\frac{n}{d^2}\right)$ we get the following
\begin{lemma}\label{Formel}For $N=\prod_{i=1}^k p_i$ with $p_1,..,p_k$ pairwise different prime numbers with $\chi_p(p_i)=1$, we have
\begin{align*} H_p^0\left(n^2\right) 
&\overset{1)}{\geq} \sum_{k=0}^{\lfloor \frac{2n}{p}\rfloor-1} H' \left(\frac{4n^2-\left(\frac{p-1}{2}+kp\right)^2}{p}\right)\\
&\overset{2)}{\geq} \frac{\pi n}{12\textnormal{e}^{\gamma}\sqrt{p}\log\log(4n^2)}\sum_{k=0}^{\lfloor \frac{2n}{p}\rfloor-1}\left(1-\left(\frac{\left(\frac{p-1}{2}\right)^2+kp(p-1)+k^2p^2}{2n^2}\right)\right)\\
&\overset{3)}{\geq} \frac{\pi n}{12\textnormal{e}^{\gamma}\sqrt{p}\log\log(4n^2)}\left(\frac{2n}{3p}-1+\frac{1}{p}\right),
\end{align*}

\end{lemma}
\begin{proof}
At first we want to write 
\begin{equation*}H_p^0(n)=\sum_{\substack{x \in \mathbb{Z}\\x^2 < 4n\\ x^2 \equiv 4n \textnormal{ mod } p}}H \left(\frac{4n-x^2}{p}\right)\end{equation*}
 in a more attainable way, namely we want to get rid of the $x$ in the formula and formulate it instead with the help of the known number $p$. But we know that

\begin{equation*}
 \frac{4n^2-x^2}{p}\geq \frac{4n^2-(\frac{p-1}{2})^2}{p},
\end{equation*}
because there exists a solution $x$ smaller than $\frac{p-1}{2}$ for the equation $x \in \mathbb{Z}, x^2 < 4n^2,  x^2 \equiv 4n \textnormal{ mod } p$.
Furthermore, one gets at least $\lfloor \frac{2n}{p}\rfloor$ solutions for the above equation, namely if $x$ is a solution, then the numbers $x+kp$ for $1 \leq k \leq \lfloor \frac{2n}{p}\rfloor-1$ are also solutions and so we get the wanted inequality by inserting them into the original formula. \\
Now we use the first inequality and Payley's formula for the resulting $d=\frac{4n^2-\left(\frac{p-1}{2}+kp\right)^2}{p}$ to get
\begin{align*}h\left(-\frac{4n^2-\left(\frac{p-1}{2}+kp\right)^2}{p}\right)
&\geq \frac{\pi}{24\textnormal{e}^{\gamma}}\frac{\sqrt{\frac{4n^2-((p-1)/2)^2+kp(p-1)+k^2p^2}{p}}}{\log(\log(4n^2))}\\ 
&\geq \frac{\pi}{24\textnormal{e}^{\gamma}}\frac{\sqrt{4n^2(1-\frac{((p-1)/2)^2+kp(p-1)+k^2p^2}{4n^2})}}{\sqrt{p}\log(\log(4n^2))}\\
&=\frac{\pi n}{12\textnormal{e}^{\gamma}}\frac{\sqrt{(1-\frac{((p-1)/2)^2+kp(p-1)+k^2p^2}{4n^2})}}{\sqrt{p}\log(\log(4n^2))}.
\end{align*}
To get rid of the square root we use the known inequality $\sqrt{1-x}\geq 1-2x$ for $0 \leq x \leq 1$. Then we get that the above term is bigger or equal to 

\begin{align*}
\frac{\pi n}{12\textnormal{e}^{\gamma}\sqrt{p}\log\log(4n^2)}\left(1-\left(\frac{\left(\frac{p-1}{2}\right)^2+kp(p-1)+k^2p^2}{2n^2}\right)\right).
\end{align*} and so the inequality 2) is proven.\\
Now part 3) is just an easy calculation of sums, but considering that $\frac{2n}{p}$ is not an integer we get an inequality
\begin{align*}
&\sum_{k=0}^{\lfloor \frac{2n}{p}\rfloor-1}\left(1-\left(\frac{\left(\frac{p-1}{2}\right)^2+kp(p-1)+k^2p^2}{2n^2}\right)\right)\\
&\geq\frac{2n}{p}-\frac{2n}{p}\left(\frac{(p-1)^2}{8n^2}\right)-\frac{p-1}{2n}\left(\frac{2n}{p}-1\right)-\frac{p}{6} \frac{\left(\frac{2n}{p}-1\right)\left(\frac{4n}{p}-1\right)}{n}\\
&=\frac{2n}{3p}-1-\frac{p}{4n}+\frac{1}{2n}-\frac{1}{pn}+\frac{p}{2n}+\frac{1}{p}-\frac{1}{2n}-\frac{p}{6n}\\
&=\frac{2n}{3p}-1+\frac{1}{p}-\frac{1}{pn}+\frac{p}{12n}\\
&\geq \frac{2n}{3p}-1+\frac{1}{p},\end{align*}
since $p \geq 5$ we have that $\frac{p}{12n}$ is bigger than $\frac{1}{pn}$ and so the proof is finished.
\end{proof}
Now we insert this estimate for $H_p^0(n^2)$ into the formula for $T_N^2$ to get:

\begin{lemma}\label{Absch} For $N=\prod_{i=1}^k p_i$ with $p_1,..,p_k$ pairwise different prime numbers with $\chi_p(p_i)=1$, we have
\begin{align*}T_N^2 \geq -\frac{1}{6}cN\log\log(N)+\frac{N\delta}{\sqrt{p}\log\log(4N^2)} \left(\frac{2N}{3p}-1+\frac{1}{p}\right),
\end{align*} with $\delta:=\frac{\pi}{12\textnormal{e}^{\gamma}}$, $c:=\textnormal{e}^{\gamma}+0.6482$.
\end{lemma} 
\begin{proof}
\begin{align*}T_N^2&=\frac{1}{2} \sum_{n|N}n \left(H_p\left(\frac{N^2}{n^2}\right)\right)\\
&= \sum_{n|N}n H_p^0\left(\left(\frac{N}{n}\right)^2\right)-\frac{1}{6} \sum_{n|N}n\\
&=-\frac{1}{6}\sigma_1(N)+\sum_{n|N}n\sum_{\substack{x \in \mathbb{Z}\\x^2 < 4(N/n)^2\\ x^2 \equiv 4(N/n)^2 \textnormal{ mod } p}}H \left(\frac{4(N/n)^2-x^2}{p}\right)\\
&\overset{\textnormal{Lemma \ref{Formel}}}{\geq} -\frac{1}{6}\sigma_1(N)+\frac{\pi N}{12\textnormal{e}^{\gamma}\sqrt{p}\log\log(4N^2)}\cdot \\ & \sum_{n|N, n\neq 1}\left(\frac{2n}{3p}-1+\frac{1}{p}\right)\\
&\geq -\frac{1}{6}\sigma_1(N)+\frac{\delta N}{6\sqrt{p}\log\log(4N^2)} \left(\frac{2N}{3p}-1+\frac{1}{p}\right),
\end{align*}
where we only consider the summand for $n=N$.\\
Furthermore, we need an estimate for $\sigma_1(N)$, but Robin showed in \cite{Robin} that for $N \geq 3$ we have $\sigma_1(N)<\textnormal{e}^{\gamma}N\log\log(N)+0.6482\frac{N}{\log\log(N)}\leq (\textnormal{e}^{\gamma}+0.6482)N\log\log(N)$.
\end{proof}
Now we can use this formula to show that for $N$ big enough the self-intersection number $T_N^2$ will always be positive.

\begin{lemma}\label{big0}
The minimum of $T_N^2$ is obtained for $N \leq p^k, k \geq \frac{3}{2(1-\eps)}$ for some $\eps$ depending on $p$ ($\eps$ can be chosen as $\frac{\log(\log(\log(p)))}{\log(p)}$).
\end{lemma}

\begin{proof}First, we know that there exists $\varepsilon > 0 $ with $\frac{1}{\log\log N} \geq N^{-\varepsilon}$ and $\varepsilon \rightarrow 0$ as $N \rightarrow \infty$.\\ \\
From Lemma \ref{Absch} we know \begin{align*}T_N^2 \geq -\frac{1}{6}cN\log\log(N)+\frac{N\delta}{\sqrt{p}\log\log(4N^2)} \left(\frac{2N}{3p}-1+\frac{1}{p}\right),
\end{align*}
inserting the estimate for $\log\log(N)$ we get
\begin{align*}T_N^2 \geq -\frac{1}{6}cN^{1+\eps}+\frac{N^{1-2\eps}\delta}{\sqrt{p}} \left(\frac{2N}{3p}-1+\frac{1}{p}\right).
\end{align*}
The leading term of the positive right side is \begin{align*}\frac{N^{1-2\eps}\delta}{\sqrt{p}}\frac{2N}{3p} ,\end{align*} so we know that the minimum has be to obtained before the exponent of $p$ of that term is bigger than the exponent of the negative term. \\
Putting $N=p^k$ and looking at the exponent we get
\begin{align*}
 k(1+\eps)&\leq k(2-2\eps)-3/2\\
 \Leftrightarrow & \frac{3}{2}\leq 2k-2\eps k - k - k\eps\\
\Leftrightarrow & \frac{3}{2} \leq k(1-3\eps)\\
\Leftrightarrow & \frac{3}{2(1-3\eps)}\leq k.
\end{align*}
\end{proof}
\begin{remark}For $\eps=\frac{1}{10}$ we have that $\frac{1}{\log\log N} \geq N^{-\varepsilon}$ is true for $N \geq 2$, so $T_{p^{\frac{15}{7}}} \geq 0$ for all $p$.
\end{remark}
\begin{lemma}For $N=\prod_{i=1}^k p_i$ with $p_1,..,p_k$ pairwise different prime numbers with $\chi_p(p_i)=1$, we have for $p \rightarrow \infty$
\begin{align*}T_N^2\geq -\frac{1}{96}\frac{c^2}{\delta}p^{\frac{3}{2}}.
\end{align*}
\end{lemma}
\begin{proof}We know by Lemma \ref{big0}  that the minimum has to be obtained for a $N$ smaller than $p^k$, so we know that for the intervall where the minimum is obtained $\log\log(N)\leq p^{k\eps}$ is true.\\
Therefore we can replace $\log\log(N)$ by $p^{k\eps}$ in the formula and get
\begin{align*}
T_N^2 \geq -\frac{1}{6}cNp^{k\eps}+\frac{N\delta}{\sqrt{p}p^{2k\eps}}  \left(\frac{2N}{3p}-1+\frac{1}{p}\right)=:t(N).
\end{align*}
To get the minimum of the self-intersection numbers we differentiate $t(N)$ and get
\begin{align*}t'(N)=-\frac{1}{6}cp^{k\eps}+\frac{\delta}{\sqrt{p}p^{2k\eps}}  \left(\frac{4N}{3p}-1+\frac{1}{p}\right).
\end{align*}
We have that this is equal to $0$ for 
\begin{align*}N_{\textnormal{min}}=\frac{1}{8}\frac{c}{\delta}p^{\frac{3}{2}}p^{3k\eps}-\frac{3}{4}(p-1).
\end{align*}
Since the second derivative of $t(N)$ is always positive, we have found our minimum.
Inserting this into the estimate for $T_N^2$, we get that
\begin{align*}
T_N^2 \geq  \left(\frac{1}{96}\frac{c^2}{\delta}\right)  p^{\frac{3}{2}}p^{4k\eps}-(p-1)cp^{k\eps}-\frac{3}{4}(p-1)-\frac{1}{8}cp^{\eps}(p-1)+\frac{3\delta}{8p^{3/2+2k\eps}}(p-1)^2
\end{align*}
For the asymptotical behaviour as $p\rightarrow \infty$ we therefore get that 
\begin{align*}\lim_{p\rightarrow \infty}t(N_{\textnormal{min}})=\left(\frac{1}{96}\frac{c^2}{\delta}\right)  p^{\frac{3}{2}}p^{4k\eps} 
\end{align*} where ${4k\eps} \rightarrow 0$ as $p$ tends to $\infty$.\\

\end{proof}

\begin{remark}If the Riemann Hypothesis is true we can take $\delta:=\frac{\pi}{6\textnormal{e}^{\gamma}}$ for all $p$ and $c:=\textnormal{e}^{\gamma}$ for all $N \geq 5041$. 
\end{remark}
\newpage

 \section{Bounded negativity for non-quaternionic Hilbert modular surfaces}

In this part I will follow closely the proof of Theorem 1 in \cite{SS_7}.\\
First, as in the proof of Theorem 1 we use the following theorem by Miyaoka (\cite{Mia})
\begin{thm}(Theorem 2.2. in \cite{SS_7})\label{alphathm} Let $X$ be a surface of nonnegative Kodaira dimension, let $C$ be an irreducible curve of geometric genus $g$ on $X$ and $K_X$ the canonical divisor. Then
\begin{align*}\frac{\alpha^2}{2}(C^2+3C\cdot K_X-6g+6)-2\alpha(C\cdot K_X-3g+3)+3c_2-K_X^2\geq 0
\end{align*}
$\forall \alpha \in [0,1]$.
\end{thm}
Let $C$ be a curve on a Hilbert modular surface $X$ of genus $g$. Then again as in \cite{SS_7} we define the difference 
\begin{align}\label{delta}\delta:= p-g=\frac{1}{2}(K_X\cdot C+C^2-2g+2),
\end{align} where $K_X$ is the canonical divisor of $X$ and $p$ is the arithmetic genus of $C$. $\delta$ is always positive.

\begin{lemma}
For a Shimura curve $C$ on $X$ we have 
\begin{align*}C^2\geq -9d_2,\end{align*}
where $d_2:=3c_2-K_X^2$
\end{lemma}
\begin{proof}
\noindent As in Theorem 3.5. in \cite{SS_7} we use the polynomial $P(\al)$ which is obtained by solving for g from (\ref{delta}) and inserting it into the equation of Theorem \ref{alphathm}
\begin{align*}
 P(\al)=\al^2(3\delta-C^2)+\al(CK_X+3C^2-6\delta)+d_2\geq 0,
\end{align*}
where $d_2:=3c_2-K_X^2$.\\
With $\rho(C)=2\left(\deg S_C-S_X\cdot C\right)$, where $S_X=\bar{X}\setminus X$ is a strict normal crossing divisor and $S_{C}=C \cap S_{X}$ is the boundary divisor we get the Hirzebruch-Höfer proportionality theorem $(K_X+S)C+2C^2+\rho(C)=4\delta$ by adjusting the proof in [9] to the non compact case. Therefore this becomes
\begin{align*}
 P(\al)=\al^2(3\delta-C^2)+\al(C^2-S\cdot C-\rho(C)-2\delta)+d_2\geq 0.
\end{align*}
the minimum of $P(\al)$ is attained for
\begin{align*}
 \alpha_0:=\frac{2\delta +S\cdot C+\rho(C)-C^2}{2(3\delta -C^2)}.
\end{align*}
Evaluating the condition $P(\al_0)\geq 0$ we get
\begin{align*}
 2d_2+2\sqrt{d_2^2+\delta d_2+d_2 S\cdot C}&\geq 2\delta+S C+\rho(C)-C^2\\
&\geq 2d_2-2\sqrt{d_2^2+\delta d_2+d_2 S\cdot C+d_2\rho(C)}.
\end{align*}
If $C^2-S\cdot C\geq 2\delta$, then $C^2\geq 0$, but 
for $C^2-S\cdot C<2\delta $ we get the lower bound
\begin{align}
 C^2&\geq 2\delta+S\cdot C+\rho(C)-2d_2-2\sqrt{d_2^2+\delta d_2+d_2 S\cdot C+d_2\rho(C)}\notag \\ &\geq 2\delta+S\cdot C+\rho(C)-2d_2-2\sqrt{d_2^2+\delta d_2}-2\sqrt{d_2 S\cdot C}-2\sqrt{d_2 \rho(C)},
\end{align}
where we used the triangle inequality in the last step.\\
\noindent For $f(x)=x-2\sqrt{d_2x}$ the minimum is achieved at $x=d_2$, because the differential is
\begin{align*}
 f'(x)=1-\sqrt{\frac{d_2}{x}},
\end{align*}
and $f(d_2)=-d_2$. Since we have this function twice, once with $x=S\cdot C$ and once with $x=\rho(C)$, the lower bound of $C^2$ by (2) becomes
\begin{align*}
 C^2\geq 2\delta-4d_2-2\sqrt{d_2^2+\delta d_2}.
\end{align*}
We can calculate that this is bigger or equal to zero for $\delta \geq \frac{5+\sqrt{13}}{2}d_2$.\\

For $\delta < \frac{5+\sqrt{13}}{2}d_2$ we have $C^2<0$ and 
\begin{align*}
 C^2\geq 2\delta-4d_2-2\sqrt{d_2^2+\delta d_2}.
\end{align*}
Since $-2\sqrt{d_2^2+\delta d_2}>-2\sqrt{d_2^2+\frac{5+\sqrt{13}}{2}d_2^2}$, we get
\begin{align*}
 C^2 \geq 2\delta-4d_2-2\sqrt{\frac{7+\sqrt{13}}{2}d_2^2},
\end{align*}
therefore we get the wanted result $C^2\geq \left(-4-2\sqrt{\frac{7+\sqrt{13}}{2}}\right)d_2 \geq -9d_2$.\end{proof}

To understand the asymptotical behaviour of $d_2(X^{\mathfrak{a}})$ we look at $c_2(X^{\mathfrak{a}})$. We know that $c_2(\bar{X^{\mathfrak{a}}})=$vol($X^{\mathfrak{a}}$)$+l(X^{\mathfrak{a}})$,
where $l(X^{\mathfrak{a}})$  comes from the cusps, and vol($X^{\mathfrak{a}})=[SL_2(\mathcal{O}):SL_2(\mathcal{O},\mathfrak{a})]2\zeta_K(-1)$ with $\zeta_K(-1)$ the 
Dedekind zeta-function. For $K$ a real quadratic field with discriminant $p$
\begin{equation*}
 \zeta_K(-1)=\frac{1}{60} \sum_{x \in \mathbb{Z}}\sigma_1\left(\frac{p-x^2}{4}\right),
\end{equation*}
where $\sigma_1(x)=0$ if $x \notin \mathbb{Z}_{\geq 1}$ and $\sigma_1(x)=\sum_{d|x}d$ if $x \in \mathbb{Z}_{\geq 1}$.\\

For $\sigma_0(x)$ which is defined as $\sigma_0(x)=0$ if $x \notin \mathbb{Z}_{\geq 1}$ and $\sigma_1(x)=\sum_{d|x}1$ if $x \in \mathbb{Z}_{\geq 1}$ we know by Lemma 5.3 in \cite{{SS_GvdG88}}:
\begin{equation*}\sum_{x \in \mathbb{Z}}\sigma_0\left(\frac{p-x^2}{4}\right)\leq p^{\frac{1}{2}}\left(\frac{3}{2\pi^2}\log^2(p)+1.05\log p\right)
\end{equation*}
for $p \equiv 1 $ mod $4$.\\
So we get 
\begin{align*} \sum_{x \in \mathbb{Z}}\sigma_1\left(\frac{p-x^2}{4}\right)&\leq \sum_{x \in \mathbb{Z}}p\sigma_0\left(\frac{p-x^2}{4}\right)\\
&\leq  p^{\frac{3}{2}}\left(\frac{3}{2\pi^2}\log^2(p)+1.05\log p\right).\end{align*}
The part coming from the cusps, namely $l(X^{\mathfrak{a}})$, splits into the part coming from the curves in the resolution of the cusps and the part coming from the quotient singularities.
For the number of curves in the resolution of the cusps, we know by chapter 4 in \cite{{SS_GvdG88}} that this is equal to
\begin{equation*}\frac{1}{2}\sum_{x \in \mathbb{Z}}\sigma_0\left(\frac{p-x^2}{4}\right)\leq \frac{1}{2}p^{\frac{1}{2}}\left(\frac{3}{2\pi^2}\log^2(p)+1.05\log p\right).
\end{equation*}
For the part of the quotient singularities we know by chapter 7 in \cite{{SS_GvdG88}} that for $p>500$ it is equal to $\frac{3}{2}a_2+\frac{5}{3}a_3^++\frac{8}{3}a_3^-$, where $a_2, a_3^+, a_3^-$ are the numbers of quotient singularities of type $(2;1,1), (3;1,1)$ and $(3;1,-1)$. By chapter 1 in \cite{{SS_GvdG88}} we know that $a_2=h(-4p), a_3^+\geq 4h\left(\frac{-p}{3}\right)$ and $a_3^-\geq \frac{1}{2}h(-3p)$, where $h(\cdot)$ is the class number. With Payley's inequality we then get that 
\begin{align*}a_2&\geq \frac{\pi}{12\textnormal{e}^{\gamma}}\frac{\sqrt{p}}{\log\log(4p)}\\
a_3^+&\geq \frac{\sqrt{3}\pi}{6\textnormal{e}^{\gamma}}\frac{\sqrt{p}}{\log\log(3p)}\\
a_3^-&\geq \frac{\sqrt{3}\pi}{48\textnormal{e}^{\gamma}}\frac{\sqrt{p}}{\log\log(3p)}.
\end{align*}
Therefore we have
\begin{align*}C^2 &\geq -\left(\frac{9}{10}p^{\frac{3}{2}}\left(\frac{3}{2\pi^2}\log^2(p)+1.05\log p\right)+ \frac{27}{2}p^{\frac{1}{2}}\left(\frac{3}{2\pi^2}\log^2(p)+1.05\log p\right)\right.\\
&+\left.\frac{27\pi}{8\textnormal{e}^{\gamma}}\frac{p^{\frac{1}{2}}}{\log\log(4p)}+\frac{3\cdot 5\sqrt{3}\pi}{2\textnormal{e}^{\gamma}}\frac{p^{\frac{1}{2}}}{\log\log(3p)}+\frac{3 \cdot \sqrt{3}\pi}{2\textnormal{e}^{\gamma}}\frac{p^{\frac{1}{2}}}{\log\log(3p)}\right).
\end{align*}

\Addresses

\end{document}